\documentclass[12pt, reqno]{amsart}
\usepackage{amsmath, amsthm, amscd, amsfonts, amssymb, graphicx, color}
\usepackage[bookmarksnumbered, colorlinks, plainpages]{hyperref}
\hypersetup{colorlinks=true,linkcolor=red, anchorcolor=green, citecolor=cyan, urlcolor=red, filecolor=magenta, pdftoolbar=true}

\newtheorem{theorem}{Theorem}[section]
\newtheorem{lemma}[theorem]{Lemma}
\newtheorem{corollary}[theorem]{Corollary}
\newtheorem{proposition}[theorem]{Proposition}
\theoremstyle{definition}
\newtheorem{definition}[theorem]{Definition}

\theoremstyle{remark}
\newtheorem{remark}[theorem]{Remark}
\numberwithin{equation}{section}
\usepackage[colorlinks]{hyperref}
\hypersetup{colorlinks=true,linkcolor=red, anchorcolor=blue,
	citecolor=blue, urlcolor=red, filecolor=magenta, pdftoolbar=true}

\title[Stability of quadratic functional equation ]{Stability of quadratic functional equation in modular spaces}

 \author[A. Baza]{Abderrahman Baza$^{1}$}
\address{$^{1}$Laboratory of Analysis, Geometry and Application, Departement of Mathematics,  University of Ibn Tofail, Kenitra, Morocco}
\email{abderrahmane.baza@gmail.com}

\author[M. Rossafi]{Mohamed Rossafi$^{2}$}
\address{$^{2}$Laboratory Analysis, Geometry and Applications, Higher School of Education and Training, University of Ibn Tofail, Kenitra, Morocco}

\email{rossafimohamed@gmail.com}

\author[A. J. Gnanaprakasam]{Arul Joseph Gnanaprakasam$^{3*}$}
\address{$^{3}$Department of Mathematics, College of Engineering and Technology, Faculty of Engineering and Technology SRM Institute of Science and Technology SRM Nagar, Kattankulathur, Kanchipuram Chennai Tamil Nadu603203, India}
\email{aruljoseph.alex@gmail.com}
\thanks{$^{*}$Corresponding author}


\begin{document}
	\begin{abstract}
		In this paper, we study the Hyers-Ulam stability of the following equation
\begin{multline*}
			\phi(x+y-z)+\phi(x+z-y)+\phi(y+z-x)=\phi (x-y)+\phi(x-z)+\phi(z-y)\\
			+\phi(x)+\phi(y) +\phi(z)
\end{multline*}
		in modular space, with or without $\Delta_2$-condition, and in $\beta$-homogeneous Banach space.
\end{abstract}
	
		\maketitle
\section{Introduction and preliminaries}
	In 1950, Nakano \cite{Nakano1950} established the theory of modulars and modular spaces. 
	After that, many autors have been perfectly developped this theory, for example, Orlicz \cite{Orlicz}, Koshi\cite{koshi}, Amemiya \cite{Amemiya}, Masur\cite{Mazur}, Luxemburg \cite{Luxemburg}, Musielak \cite{Musielak}. 
	The study of various Orlicz spaces \cite{Orlicz} and interpolation theory \cite{Gavruta,Ulam} have been  extensively applied the theory of modulars and modular sapaces.
	Firstly, we introduce the principal definitions, notations and porperties of the theory of modular and modular spaces.
	\begin{definition}\label{Definition1.1}
	Let $Y$ be an arbitrary vector space. A functional $\rho: Y \rightarrow[0, \infty)$ is called a modular if for arbitrary $u, v \in Y$;
	\begin{enumerate}
		\item $\rho(u)=0$ if and only if $u=0$.\label{item1}
		\item $\rho(\alpha u)=\rho(u)$ for every scalar $\alpha$ with $|\alpha|=1$.
		\item $\rho(\alpha u+\beta v) \leq \rho(u)+\rho(v)$ if and only if $\alpha+\beta=1$ and $\alpha ,\beta \geq 0$.\\
		If \eqref{item1} is replaced by:
		\item $\rho(\alpha u+\beta v) \leq \alpha \rho(u)+\beta \rho(v)$ if and only if $\alpha+\beta=1$ and $\alpha, \beta \geq 0$, then we say that $\rho$ is a convex modular.
	\end{enumerate}
	A modular $\rho$ defines a corresponding modular space, i.e., the vector space $Y_\rho$ given by:
	\begin{equation*}
		Y_\rho=\{u \in Y: \rho(\lambda u) \rightarrow 0 \text { as } \lambda \rightarrow 0\}.
	\end{equation*}
	A function modular is said to be satisfy the $\Delta_2$-condition $\rho$  if there exist $\tau > 0$ such that $\rho(2 u) \leq \tau \rho(u)$ for all $u \in Y_\rho$.
\end{definition}
\begin{definition}
	Let $\left\{u_n\right\}$ and $u$ be in $Y_\rho$. Then:
	\begin{enumerate}
		\item 
		The sequence $\left\{u_n\right\}$, with $u_n \in Y_\rho$, is $\rho$-convergent to $u$ and write: $u_n \to  u$ if $\rho\left(u_n-u\right)\to 0$ as $n \rightarrow \infty$.
		\item
		The sequence $\{u_n\}$, with $u_n \in Y_\rho$, is called $\rho$-Cauchy if\\ $\rho\left(u_n-u_m\right)\to 0$ as $n: m \to \infty$.
		\item
		$Y_\rho$ is called $\rho$-complete if every $\rho$-Cauchy sequence in $Y_\rho$ is $\rho$-convergent.
	\end{enumerate}
\end{definition}
\begin{proposition}
	In modular space,
	\begin{itemize}
		\item If $u_n \overset{\rho}{\to} u$ and a is a constant vector, then $u_n+a \overset{\rho}{\to} u+a$.
		\item If $u_n \overset{\rho}{\to} u$ and $v_n \overset{\rho}{\to} v$ then $\alpha u_n + \beta v_n \overset{\rho}{\to} \alpha u+ \beta v$, where $\alpha+\beta \leq 1$ and $\alpha,\beta \geq 1$. 
	\end{itemize}
\end{proposition}

\begin{remark}
	Note that $\rho(u)$ is an increasing function, for all $u \in X$. 
	Suppose $0<a<b$, then property $(4)$ of Definition \ref{Definition1.1} with $v=0$ shows that $\rho(a u)=\rho\left(\dfrac{a}{b} b u\right) \leq \rho(b u)$ for all $u \in Y$. 
	Morever, if $\rho$ is a convexe modular on $Y$ and $|\alpha| \leq 1$, then $\rho(\alpha u) \leq \alpha \rho(u)$.
	
	In general, if $\lambda_i \geq 0$, $i=1, \dots,n$ and  $\lambda_1,\lambda_2,\dots,\lambda_n \leq1$ then $\rho (\lambda_1 u_1+\lambda_2 u_2+\dots+\lambda_n u_n) \leq \lambda_1 \rho(u_1)+\lambda_2 \rho(u_2)+\dots+\lambda_n \rho(u_n)$.
	
	If $\{u_n\}$ is $\rho$-convergent to $u$, then $\{ c u_n \}$ is $\rho$-convergent to $cu$, where $|c| \leq 1$.
	But the $\rho$-convergent of a sequence $\{u_n\}$ to $u$ does not imply that $\{\alpha u_n\}$ is $\rho$-convergent to $\alpha u_n$ for scalars $\alpha $ with $|\alpha|>1$.
	
	If $\rho$ is a convex modular satisfying $\Delta_2$ condition with $0 <\tau<2$, then $\rho(u) \leq \tau \rho(\dfrac{1}{2} u) \leq \dfrac{\tau}{2} \rho(u)$ for all $u$.\\
	Hence $\rho=0$. Consequently, we must have $\tau \geq 2$ if $\rho$ is convex modular. 
\end{remark}

The problem of stability has been released by Hyers in 1941 \cite{Hyers}, when he gived an answer to the question: There exists a homomorphrism  $F: G_1 \rightarrow G_2$ such that $d(f(x), F(x))<\varepsilon$ for all $x \in G_1$, with $G_1$ is a group and $(G_2, d)$ is a metric group and $f: G_1 \rightarrow G_2$ is a mapping satisfying $d(f(x y), f(x) f(y))<\delta$, $\varepsilon, \delta>0?$.
	
	M. Rassias \cite{Rassias} developped Hyers theorem of stability, when he studied the stability of additive fuctional equatuon
	$$
	f(x+y)=f(x)+f(y).
	$$
	A Generalization of stability has been done by Gavrüta \cite{Gavruta} in 1994, anol ofter that, many anthors have been extensively studied the stability problems of functional equations and inequations, and there are a number of results concerning this problem. \cite{Orlicz,Hyers,Jung}
	
	In this paper witch is made up of 4 sections, we study the Hyere-Ulam stability of the following equation
	\begin{multline}\label{Eq-1}
				\phi(x+y-z)+\phi(x+z-y)+\phi(y+z-x)=\phi (x-y)+\phi(x-z)+\phi(z-y)\\+\phi(x)+\phi(y) +\phi(z)
	\end{multline}
	in section 2, using the direct method, we study the stability of \eqref{Eq-1} in modular space without $\Delta_2$-condition. In section 3, we study the statrility of \eqref{Eq-1} in modular space satisfying $\Delta_2$-condition. In section 4, we study the stability of \eqref{Eq-1} in $\beta$-homogeneous complex Banach space.
	\section{Stability of \eqref{Eq-1} in modular spaces without $\Delta_2$-condition}
	\begin{lemma}[\cite{Lan-2015}]\label{Lemma2.1}
		Let $\phi: X \rightarrow Y$ be a mapping satisfying \eqref{Eq-1}. Then $\phi$ is quadratic, and there exists a symmetric biadditive mapping:
		$A: X \times X \rightarrow Y$ such that: $f(x)=A(x, x)$ for all $x \in X$.
	\end{lemma}
	
	\begin{theorem}\label{Theorem2.1}
		Let $X$ be a vector space,
		 $Y_\rho$ be a $\rho$-complete convex modular space. Let $ \alpha$: $X^3 \rightarrow[0, \infty)$ be a function such that:
		
		\begin{equation}\label{Eq-2}
			\varphi(x, y, z)=\sum_{j=1}^{\infty} \frac{1}{4^j} \alpha\left(2^{j-1} x, 2^{j-1} y, 2^{j-1} z\right)<\infty , \text{ for all } x,y,z \in X.
		\end{equation}
		Let $\phi: X \rightarrow  Y_\rho $ a mapping satisfying $\phi(0)=0$ and:
		\begin{multline}\label{Eq-3}
			\rho(\phi(x+y-z)+\phi(x+z-y)+\phi(y+z-x)-\phi(x-y)-\phi(x-z)-\phi(z-y) \\ -\phi(x)-\phi(y)
			-\phi(z)) \leq \alpha(x, y, z),
		\end{multline} for all $x, y, z \in X$.
		Then there exists a unique quadratic mapping $h: X \rightarrow Y$ such that: 
		\begin{equation}\label{Eq-4}
			\rho(\phi(x)-h(x)) \leq \varphi(x, x, 0)
		\end{equation}
		and: 
		$$ h(x)=\rho-\operatorname{limit} \frac{\phi\left(2^n x\right)}{4^n} ; \qquad x \in X.$$
	\end{theorem}

	\begin{proof}
		Letting $x=y$ and $z=0$ in \eqref{Eq-3}, we get
		$$
		\rho(\phi(2 x)-4 \phi(x)) \leq \alpha(x, x, 0).
		$$
		Then 
		\begin{equation}\label{Eq-5}
			\rho\left(\frac{1}{4} \phi(2 x)-\phi(x)\right) \leq \frac{1}{4} \alpha(x, x, 0).
		\end{equation}
		by simple induction, we have:
		\begin{equation}\label{Eq-6}
			\rho\left(\frac{1}{4^k} \phi\left(2^k x\right)-\phi(x)\right) \leq \sum_{j=1}^k \frac{1}{4^j} \alpha\left(2^{j-1} x, 2^{j-1} x, 0\right)
		\end{equation}
		for all $x \in X$, and all positif integer $k$.
		\\
		For $k=1$, we obtain \eqref{Eq-5}. 
		Suppose that \eqref{Eq-6} holds for $k \in \mathbb{N}$. 
		We have:
		\begin{align*}
			\rho\left(\frac{1}{4^{k+1}} \phi\left(2^{k+1} x\right)-\phi (x)\right)&=\rho\left(\frac{1}{4}\left(\frac{\phi\left(2^k \cdot 2 x\right)}{4^k}-\phi(2 x)\right)+\frac{1}{4} \phi(2 x)-\phi(x)\right) \\
			&\leq \frac{1}{4} \rho\left(\frac{\phi\left(2^k. 2 x\right)}{4^k}-\phi(2 x)\right)+\frac{1}{4} \rho(\phi(2 x)-4 \phi(x)) \\
			&\leq \sum_{j=1}^k \frac{1}{4^{j+1}} \alpha\left(2^j x, 2^j x, 0\right)+\frac{1}{4} \alpha(x, x, 0) \\
			&=\sum_{j=1}^{k+1} \frac{1}{4^j} \alpha\left(2^{j-1} x, 2^{j-1} x, 0\right)
		\end{align*}
		Thus, \eqref{Eq-6} holds for every $k \in\mathbb{N}$.\\
		Let $m, n$ be positive integers with $n>m$, we write
		\begin{align}
			\rho\left(\frac{\phi\left(2^n x\right)}{4^n}-\frac{\phi\left(2^m x\right)}{4^m}\right) & =\rho\left(\frac{1}{4^m}\left(\frac{\phi\left(2^{n-m} \cdot 2^m x\right)}{4^{n-m}}-\phi\left(2^m x\right)\right)\right. \nonumber\\
			& \leq \frac{1}{4^m} \sum_{j=1}^{n-m} \frac{1}{4^j} \alpha\left(2^{m+j-1}x, 2^{m+j-1} x, 0\right) \nonumber\\
			& =\sum_{k=m+1}^n \frac{1}{4^k} \alpha\left(2^{k-1} x, 2^{k-1} x, 0\right) \label{Eq-7}
		\end{align}
		it follows from \eqref{Eq-7} and \eqref{Eq-2} that the sequence $\left\{\frac{\phi(2^n x)}{4^n}\right\}$ is a $\rho$-Cauchy sequence in $Y_p$. 
		Since $Y_\rho $ is $\rho$-complete, then the sequence $\left\{\frac{\phi\left(2^n x \right)}{4^n}\right\}$ is $\rho$-convergent to $h(x)$. 
		Thus, we write
		\begin{equation}\label{Eq-8}
		h(x)=	\rho-\lim \frac{\phi\left(2^n x\right)}{4^n}; x \in X .
		\end{equation}
		Now, we see that 
		\begin{align*}
			\rho\left(\frac{h(2 x)-4 h(x)}{4^3}\right)&=\rho\left(\frac{1}{4^3}\left(h(2 x)-\frac{\phi\left(2^{n+1} x \right)}{4^n}\right)+\frac{1}{4}\left(\frac{1}{4} \frac{\phi\left(2^{n+1} x \right)}{4^{n+1}}-\frac{1}{4} h(x)\right)\right) \\
			& \leq \frac{1}{4^3} \rho \left(h(2 x)-\frac{\phi\left(2^{n+1} x\right)}{4^n}\right)+\frac{1}{16} \rho\left(\frac{\phi\left(2^{n+1} x \right)}{4^{n+1}}-h(x)\right)\\
			&		\to 0 \text{ as }n \to \infty.
		\end{align*}
		for all  $x \in X$.
	\\
		Then
		\begin{equation}\label{Eq-9}
			h(2 x)=4 h(x).
		\end{equation}
		Next, we have 
		\begin{align*}
			\rho(h(x)-\phi(x)) & =\rho\left(\sum_{k=1}^n \frac{\phi\left(2^k x\right) - 4 \phi\left(2^{k-1} x\right)}{4^k}+\left(h(x)-\frac{\phi\left(2^n x\right)}{4^n}\right)\right) \\
			& =\rho \left(\sum_{k=1}^n \frac{\phi\left(2^k x\right)-4 \phi\left(2^{k-1} x\right)}{4^k}+\frac{1}{4}\left(h(2 x)-\frac{\phi\left(2^{n-1} \cdot 2 x\right)}{4^{n-1}}\right)\right)
		\end{align*}
		Since 
		$$\sum_{k=1}^n \frac{1}{4^k}+\frac{1}{4}<1,$$ 
		it follows that:
		\begin{align}
			\rho(h(x)-\phi(x)) &\leq \sum_{k=1}^n \frac{1}{4^k} \rho \left(\phi\left(2^k x\right)-4 \phi\left(2^{k-1} x\right)\right)+\frac{1}{4} \rho \left(h(2 x)-\frac{\phi\left(2^{n-1} 2 x\right)}{4^{n-1}}\right) \nonumber \\
			&\leq \sum_{k=1}^n \frac{1}{4^k} \alpha\left(2^{k-1} x, 2^{k-1} x, 0\right)+\frac{1}{4} \rho\left(h(2 x)-\frac{\phi\left(2^{n-1} 2 x\right)}{4^{n-1}}\right) \label{Eq-10}
		\end{align}
		Letting $n \rightarrow \infty$ in \eqref{Eq-10}, we get:
		$$
		\rho(h(x)-\phi(x)) \leq \varphi(x, x, 0).
		$$
		Then, we arrive at \eqref{Eq-4}.\\
		Now, we prove that $h$ is quadratic. 
		Firstly, we show  that:
		\begin{align}
			&\rho\left(\frac{\phi\left(2^j(x+y-z)\right)}{4^j}+\frac{\phi\left(2^j(x+z-y)\right)}{4^j}+\frac{\phi\left(2^j(y+z-x)\right)}{4^j} \right. \nonumber\\
			&\left. \qquad- \frac{\phi\left(2^j(x-y)\right)}{4^j}-\frac{\phi\left(2^j(x-z)\right)}{4^j}
			 -\frac{\phi\left(2^j(z-y)\right)}{4^j}-\frac{\phi\left(2^j x\right)}{4^j}-\frac{\phi\left(2^j y\right)}{4^j}-\frac{\phi\left(2^j z\right)}{4^j}\right)\nonumber\\
			&\leq \frac{1}{4^j} \rho\left(\phi\left(2^j(x+y-z)\right)+\phi\left(2^j(x+z-y)\right)+\phi\left(2^j(y+z-x)\right)\right. \nonumber\\
			&\qquad\left.-\phi\left(2^j(x-y)\right)- \phi\left(2^j(x-z)\right) - \phi\left(2^j(z-y)\right)-\phi\left(2^j x\right)-\phi\left(2^j y\right)-\phi\left(2^j z\right)\right) \nonumber\\
			&\leq \frac{1}{4^j} \alpha\left(2^j x, 2^j y, 2^j z\right) \rightarrow 0 \text { as } j \rightarrow \infty .   \label{Eq-11}
		\end{align}
		by \eqref{Eq-11}, we have the following inequality:
		\begin{multline*}
		 \rho\left(\frac{1}{12}(h(x+y-z)+h(x+z-y)+h(y+z-x)-h(x-y)-h(x-z)-h(z-y)\right.\\
		 \left.  -h(x)-h(y)-h(z))\right) \\
		\leq \frac{1}{12}\left[\rho\left(h(x+y-z)-
		\frac{\phi\left(2^j(x+y-z)\right)}{4^j}\right)
	+\rho\left(h(x+z-y)-\frac{\phi\left(2^j(x+z-y)\right)}{4^j}\right) 	\right.\\
	\left.
	+\rho\left(h(y+z-x)-\frac{\phi\left(2^j(y+z-x)\right)}{4^j}\right) 
	+\rho\left(h(x-y)-\frac{\phi\left(2^j(x-y)\right)}{4^j}\right)\right. \\\left. 
	+\rho\left(h(x-z)-\frac{\phi\left(2^j(x-z)\right)}{4^j}\right)
	+\rho\left(h(z-y)-\frac{\phi\left(2^j(z-y)\right)}{4^j}\right) \right.\\
	\left.
	+\rho\left(h(x)-\frac{\phi(2^j x)} {4^j}\right)
	+\rho\left(h(y)-\frac{\phi(2^j y)}{4^j}\right)
	+\rho\left(h(z)-\frac{\phi(2^j z)}{4^j}\right)\right]\\
	+\frac{1}{12} \rho\left(\frac{\phi(2^j(x+y-z))}{4^j}+\frac{\phi\left(2^j(x+z-y)\right)}{4^j}
	+\frac{\phi\left(2^j(y+z-x)\right)}{4^j}-\frac{\phi\left(2^j(x-y)\right)}{4^j}\right. \\
	\left.-\frac{\phi(2^j (x-z))}{4^j}-\frac{\phi\left(2^j(z-y)\right)}{4^j}-
	\frac{\phi(2^j x)}{4^j} - \frac{\phi(2^j y)}{4^j} - \frac{\phi(2^j z)}{4^j}\right) 
	\rightarrow \text { as } j \rightarrow \infty.
	\end{multline*}
		Then we get:
		$$
		h(x+y-z)+h(x+z-y)+h(y+z-x)-h(x-y)-h(x-z)-h(z-y)-h(x)-h(y)-h(z)=0
		$$
		Hence by Lemma \ref{Lemma2.1}, we conclude that $h$ is quadratic.\\
		Finally for proving the uniqueness of $h$, letting $h_1$ and $h_2$ two quadratic mapping satisfying \eqref{Eq-4}.
		\\
		We have:
		\begin{align*}
			\rho\left(\frac{h_1(x)-h_2(x)}{2}\right)&=\rho\left(\frac{1}{2}\left(\frac{h_1\left(2^k x\right)}{4^k}-\frac{\phi\left(2^kx \right)}{4^k}\right)+ \frac{1}{2}\left(\frac{\phi\left(2^k x\right)}{4^k}-\frac{h_2 (2^kx)}{4^k}\right)\right) \\
			& \leq \frac{1}{2} \rho\left(\frac{h_1(2^k x)}{4^k}-\frac{\phi\left(2^k x\right)}{4^k}\right)+\frac{1}{2} \rho\left(\frac{\phi\left(2^k x\right)}{4^k}-\frac{h_2\left(2^k x\right)}{4^k}\right) \\
			& \leq \frac{1}{2} \cdot \frac{1}{4^k}\left\{\rho\left(h_1\left(2^k x\right)-\phi\left(2^k x\right)\right)+\rho\left(\phi\left(2^k x\right)-h_2\left(2^k x\right)\right)\right\} \\
			& \leq \frac{1}{4^k} \varphi\left(2^k x, 2^k x, 0\right) \\
			& =\sum_{l=k+1}^{\infty} \frac{1}{4^l} \alpha\left(2^{l-1} x, 2^{l-1} x, 0\right)
		\end{align*}
		$\rightarrow 0 \text { as } k \rightarrow \infty$.
		Then $h_1=h_2$ and this complete the proof. 
	\end{proof}
	Now, if we put $\alpha=\varepsilon>0$, we obtain a classical Ulam stability of \eqref{Eq-1}.
	
	\begin{corollary}
		Let $X$ be a vector space and $ Y_\rho $ be a $\rho$-complete convexe modular space. 
		Let $\phi: X \rightarrow  Y_\rho $ a mapping satisfying $\phi(0)=0$ and
		$$
		\rho(\phi(x+y-z)+\phi(x+z-y)+\phi(y+z-x)-\phi(x-y)-\phi(x-z)-\phi(z-y)-\phi(x)-\phi(y)-\phi(z)) \leq \varepsilon
		$$
		for all $x, y, z \in X$. Then there exists a unique quadratic mapping $h: X \rightarrow  Y_\rho $ such that:
		$$
		\rho(\phi(x)-h(x)) \leq \frac{\varepsilon}{3} ; x \in X .
		$$
	\end{corollary}
	
	\begin{corollary}
		Let $X$ be a vector space and $ Y_\rho $ be a $\rho$-complete convex modular space. 
		Let $\theta>0$ and $0<p<2$ be real numbers, and $\phi: X \rightarrow  Y_\rho $ is a mapping satisfying:
		\begin{multline*}
			\rho(\phi(x+y-z)+\phi(x+z-y)+\phi(y+z-x)-\phi(x-y)-\phi(x-z)-\phi(z-y)-\phi(x)-\phi(y) \\
			-\phi(z)) \leq \theta\left(\|x\|^p+\|y\|^p+\|z\|^p\right) \text {, for all  } x,y,z\in X.
		\end{multline*}
		Then there exists a unique quadratic mapping $h: X \rightarrow  Y_\rho $ such that:
		$$
		\rho(\phi(x)-h(x)) \leq \frac{2 \theta\|x\|^p}{2-2^p}, \qquad x \in X .
		$$
	\end{corollary}
	\section{Stability of \eqref{Eq-1} in modular space satisfying $\Delta_2$-condition }
	\begin{theorem}\label{Theorem3.1}
		Let $X$ be a vector space, $ Y_\rho $ be a $\rho$-complete convex modular space  satisfying $\Delta_2$-condition. 
		If there exists a function: $\alpha: X^3 \rightarrow[0, \infty)$ such that $\alpha(0,0)=0$ and 
		\begin{enumerate}
			\item[(i)] 	
			\begin{equation}\label{Eq-12}
				\varphi(x, y, z)=\sum_{j=1}^{\infty}\left(\frac{\tau^3}{2}\right)^j \alpha\left(\frac{x}{2^j}, \frac{y}{2^j}, \frac{z}{2^j}\right)<\infty 
			\end{equation}
			and 
			\begin{equation}\label{Eq-13}
				\lim _{n \rightarrow \infty} \tau^{2 n} \alpha\left(\frac{x}{2^n}, \frac{y}{2^n}, \frac{z}{2^2}\right) = 0,
			\end{equation}
			\item[(ii)] 
			\begin{multline}\label{Eq-14}
				\rho (\phi(x+y-z)+\phi(x+z-y)+\phi(y+z-x)-\phi(x-y)-\phi(x-z)-\phi(z-y)\\
				-\phi(x)-\phi(y)-\phi(z)) \leq \alpha(x, y, z)
			\end{multline}
			for all $x, y, z \in X$. Then there exists a unique quadratic mapping $h: X \rightarrow Y_\rho$ such that 
		\end{enumerate}
		\begin{equation}\label{Eq-15}
			\rho(\phi(x)-h(x)) \leq \frac{1}{2 \tau}  \varphi(x, x, x)
			\text{ for all } x \in X.
		\end{equation}
	\end{theorem}
	
	\begin{proof}
		Letting $x=y$ and $z=0$ in \eqref{Eq-14}, after that, we replace $x$ with $\frac{x}{2}$, we get 
		$$\rho\left(4 \phi\left(\frac{x}{2}\right)-\phi(x)\right) \leq \alpha\left(\frac{x}{2}, \frac{x}{2}, \frac{x}{2}\right)$$
		for all $x \in X$. 
		Now, we have the following expression
		\begin{align*}
			\rho \left(4^n \phi \left(\frac{x}{2^n}\right)-\phi(x)\right) & = \rho \left(\sum_{j=1}^n \frac{1}{2^j}\left(2^{3 j-2} \phi\left(\frac{x}{2^{j-1}}\right)-2^{3 j} \phi\left(\frac{x}{2^j}\right)\right)\right) \\
			& \leq \frac{1}{\tau^2} \sum_{j=1}^n\left(\frac{\tau^3}{2}\right)^j \alpha\left(\frac{x}{2^j}, \frac{x}{2^j}, \frac{x}{2^j}\right)
		\end{align*}
		for all $x \in X$.
		
		Let $m$ and $n$ be a positive integers. We have
		\begin{align*}
			\rho\left(4^m \phi\left(\frac{x}{2^m}\right) - 4^{n+m}  \phi\left(\frac{x}{2^{n+m}}\right)\right) &\leq \tau^{2 m} \rho \left(\phi\left(\frac{x}{2^m}\right)-4^n \phi\left(\frac{x}{2^{n+m}}\right)\right) \\
			& \leq \tau^{2 m-2} \sum_{j=1}^n\left(\frac{\tau^3}{2}\right)^j \alpha\left(\frac{x}{2^{j+m}}, \frac{x}{2^{j+m}}, \frac{x}{2^{j+m}}\right)\\
			&\leq \frac{2^m}{\tau^{m+2}} \sum_{l=m+1}^{n+m} \left(\frac{\tau^3}{2}\right)^l \alpha\left(\frac{x}{2^l}, \frac{x}{2^l}, \frac{x}{2^l}\right)
		\end{align*}
		$\to 0 \text{ as } m \to \infty.$\\
		(because $ \dfrac{2}{\tau} \leq 1$).\\ 
		Hence, the sequence $\left\{4^n \phi\left(\frac{x}{2^n}\right)\right\}$ is a $\rho$-Cauchy  sequence in $Y_\rho$ wich is $\rho$-complete. 
		Thus we define a mapping $h: X \rightarrow  Y_\rho $ as
		$$
		h(x)=\rho-\operatorname{limit}_{n \rightarrow \infty} 4^n \phi\left(\frac{x}{2^n}\right) \text {, for all } x \in X . 
		$$
		Now, we prove the estimation \eqref{Eq-15}. 
		We have
		\begin{align*}
			\rho(\phi(x)-h(x)) & \leq \frac{1}{2} \rho\left(2 \phi(x)-2 \cdot 4^n \phi\left(\frac{x}{2^n}\right)\right)+\frac{1}{2} \rho\left(2 \cdot 4^n \phi\left(\frac{x}{2^n}\right)-2 h(x)\right) \\
			& \leq \frac{\tau}{2} \rho\left(\phi(x)-4^n \phi\left(\frac{x}{2^n}\right)\right)+\frac{\tau}{2} \rho\left(4^n \phi\left(\frac{x}{2^n}\right)-h(x)\right) \\
			& \leq \frac{1}{2 \tau} \sum_{j=1}^n\left(\frac{\tau^3}{2}\right)^j \alpha\left(\frac{x}{2^j}, \frac{x}{2^j}, \frac{x}{2^j}\right)+\frac{\tau}{2} \rho\left(4^n \phi\left(\frac{x}{2^n}\right)-h(x)\right)
		\end{align*}
		for all $x \in X$.
		
		Since 
		$$\lim _{n \rightarrow \infty} \rho\left(4^n \phi\left(\frac{x}{2^n}\right)-h(x)\right)=0\text{ for all }x \in X,$$
		We get when $n \rightarrow \infty$
		$$
		\rho(\phi(x)-h(x)) \leq \frac{1}{2 \tau} \varphi(x, x, x) \text { for all } x \in X .
		$$
		On the other hand, we have:
		\begin{multline*}
			\rho\left(4^n \phi\left(\frac{1}{2^n}(x+y-z)\right)+4^n \phi\left(\frac{x+z-y}{2^n}\right)+4^n \phi\left(\frac{y+z-x}{2^n}\right)\right. \\
			\left.
			-4^n \phi\left(\frac{x-y}{2^n}\right)-4^n \phi\left(\frac{x-z}{2^n}\right)-4^n \phi\left(\frac{y- z}{2^n}\right)\right. \\
			\left.-4^n \phi\left(\frac{x}{2^n}\right)-4^n \phi\left(\frac{y}{2^n}\right)-4^n \phi\left(\frac{z}{2^n}\right)\right) \leq \tau^{2 n} \alpha\left(\frac{x}{2^n}, \frac{y}{2^n}, \frac{z}{2^n}\right) \rightarrow 0 \text { as } n \rightarrow \infty .
		\end{multline*}
		and we have
		\begin{multline*}
			\rho\left(\frac{1}{12} h(x+y-z)+\frac{1}{12} h(x+z-y)+\frac{1}{12} h(y+z-x)\right.\\
			\left.-\frac{1}{12} h(x-y)-\frac{1}{12} h(x-z)-\frac{1}{12} h(z-y) 
			-\frac{1}{12} h(x)-\frac{1}{12} h(y)-\frac{1}{12} h(z)\right) \\
			\leq \frac{1}{12} \rho\left(h(x+y-z)-4^n \phi\left(\frac{x+y-z}{2^n}\right)\right)+\frac{1}{12} \rho\left(h(x+z-y)-4^n \phi\left(\frac{x+z-y}{2^n}\right)\right) \\
			+\frac{1}{12} \rho\left(h(y+z-x)-4^n \phi\left(\frac{y+z-x}{2^n}\right)\right)+\frac{1}{12} \rho\left(h(x-y)-4^n \phi\left(\frac{x-y}{2^n}\right)\right) \\
			+\frac{1}{12} \rho\left(h(x-z)-4^n \phi\left(\frac{x-z}{2^n}\right)\right)+\frac{1}{12} \rho\left(h(z-y)-4^n \phi\left(\frac{z-y}{2^n}\right)\right)
		\end{multline*}

		\begin{multline*}			+\frac{1}{12} \rho\left(h(x)-4^n \phi\left(\frac{x}{2^n}\right)\right)+\frac{1}{12} \rho\left(h(y)-4^n \phi\left(\frac{y}{2^n}\right)\right)+\frac{1}{12} \rho\left(h(z)-4^n \phi\left(\frac{z}{2^n}\right)\right) \\			
		+\frac{1}{12} \rho \left(4^n \phi\left(\frac{x+y-z}{2^n}\right)+4^n \phi\left(\frac{x+z-y}{2^n}\right)+ 4^n\phi\left(\frac{y+z-x}{2^n}\right)\right. \\
			\left.
			+ 4^n\phi\left(\frac{x-y}{2^n}\right)+ 4^n\phi\left(\frac{x-z}{2^n}\right)
			+4^n \phi\left(\frac{z-y}{2^n}\right)
			-4^n\phi\left( \dfrac{x}{2^n}\right)
			-4^n\phi\left( \dfrac{y}{2^n}\right)
			-4^n\phi\left( \dfrac{z}{2^n}\right) \right)\\
			\rightarrow 0 \text { as } n \rightarrow \infty
		\end{multline*}
		Hence 
		$$h(x+y-z)+h(x+z-y)+h(y+z-x)-h(x-y)-h(x-z)-h(z-y)-h(x)-h(y)- h(z)=0,$$
		for all $x, y, z \in X$,
		and by Lemma \ref{Lemma2.1}, we deduce that $h$ is quadratic.
		
		Finally, for proving the uniqueness of $h$, we suppose that there exists another mapping $Q: X \rightarrow  Y_\rho $ (quadratic) satisfying \eqref{Eq-15}, we have:
		\begin{align*}
			\rho(h(x)-Q(x)) & \leq \frac{1}{2} \rho\left(2 \cdot 4^n h\left(\frac{x}{2^n}\right) - 2 \cdot 4^n \phi\left(\frac{x}{2^n}\right)\right)+\frac{1}{2} \rho\left(2 \cdot 4^n \phi\left(\frac{x}{2^n}\right)-2 \cdot 4^n Q\left(\frac{x}{2^n}\right)\right) \\
			& \leq \frac{\tau^{2 n+1}}{2} \rho \left(h\left(\frac{x}{2^n}\right) - \phi\left(\frac{x}{2^n}\right)\right) + \frac{\tau^{2 n+1}}{2} \rho\left(\phi\left(\frac{x}{2^n}\right)-Q\left(\frac{x}{2^n}\right)\right)\\
			& \leq \frac{\tau^{2 n}}{2} \varphi\left(\frac{x}{2^n}, \frac{x}{2^n}, \frac{x}{2^n}\right) \\
			& =\frac{2^{n-1}}{\tau^n} \sum_{l=n+1}^{\infty}\left(\frac{\tau^3}{2}\right)^l \alpha \left(\frac{x}{2^l}, \frac{x}{2^l}, \frac{x}{2^l}\right) 
			\rightarrow 0 \text { as } n \rightarrow \infty.
		\end{align*}
		Then $h(x)=Q(x)$ for all $x \in X$ and this complete the proof.
	\end{proof}		
		\begin{corollary}
		Let $X$ be a vector space, $ Y_\rho $ be a $\rho$-complete convex modular space satisfying $\Delta_2$-condition. 
		Let $r> \log _2\left(\frac{\tau^3}{2}\right)$ and $\theta>0$, be real numbers, and $\phi: X \rightarrow Y_\rho$ be
		a mapping such that
		\begin{multline*}
			\rho(\phi(x+z-z)+\phi(x+z-y)+\phi(y+z-x)-\phi(x-y)-\phi(x-z)-\phi(z-y)\\-\phi(x)-\phi(y)-\phi(z))) 
			\leq \theta\left(\|z\|^r+\|y\|^r+\|z\|^r\right)
		\end{multline*}
		for all $x, y, z \in X$. Then there exists a unique quadratic mapping $h: X \rightarrow  Y_\rho $ such that:
		$$
		\rho(\phi(x)-h(x)) \leq \frac{3 \theta \tau^2}{2\left(2^{r+1}-\tau^3\right)}\|x\|^r ; \qquad x \in X.
		$$
\end{corollary}		

\section{Stability of \eqref{Eq-1} in $\beta$-homogeneous Banach space}
\begin{definition}
	Let $X$ be a linear space over $\mathbb{C}$. The application $\|\cdot\|: X \rightarrow[0, \infty)$ is an $F$-norm if 
	\begin{enumerate}
		\item $\|x\|=0$ if and only if $u=0$.
		\item $\|\alpha u\|=\|u\|$ for avery $u \in X$ and every $\alpha$ with $|\alpha|=1$.
		\item $\|u+v\| \leq \|u\|+\|v\|$ for every $u,v \in X$.
		\item $\left\|\alpha_n u\right\| \rightarrow 0$ implies $\alpha_n \rightarrow 0$.
		\item  $\left\|\alpha u_n\right\| \rightarrow 0$ implies $u_n \rightarrow 0$.
	\end{enumerate}
	Letting $d(u, v)=\|u-v\|$. So $(X,d)$ is a metric space which called $F$-space if $d$ is complete.\\
	If $\|\alpha u\|= |\alpha|^\beta\|u\|$ for all $u \in X$ and $\alpha \in \mathbb{C}$, then $\| \cdot\| $ is called $\beta$-homogeneous $(\beta>0)$.\\
	A $\beta$-homogeneous $F$-space is called a $\beta$-homogenous Banach space.
\end{definition}
\begin{remark}
	If $\rho$ is a convex modular, then
	$$
	\|u\|_\rho=\inf \left\{\lambda^k>0 / \rho\left(\frac{u}{\lambda}\right) \leq  1\right\}, u \in Y_\rho 
	$$
	is an $F$-norm on $Y_\rho $ that satisfies $\|\alpha u\|_\rho=|\alpha|^{k}\|u\|_\rho$.
	Hence, $\|\cdot\|_\rho$ is $k$-homogeneous. 
	In the case $k=1$, this norm is called the Luxembourg norm.
\end{remark}
	
	\begin{theorem}
		Let $X$ be a vector space, $Y$ be a $\beta$-homogeneous complex Banach spaces and $\alpha: X^3 \rightarrow[0, \infty)$ be a function such that:
		\begin{equation}\label{Eq-4.1}
			\varphi(x, y, z)=\sum_{j=1}^{\infty} \frac{1}{4^{\beta j}} \alpha\left(2^{j-1} x, 2^{j-1} y, 2^{j-1} z\right)<\infty .
		\end{equation}
		for all $x, y, z \in X.$
		Let $\phi: X \rightarrow Y$ be a mapping satisfying $\phi(0)=0$ and
		\begin{multline}\label{Eq-12}
			\| \phi(x+y-z)+\phi(x+z-y)+\phi(y+z-x)-\phi(x-y)-\phi(x-z)-\phi(z-y) \\ -\phi(x)-\phi(y)-\phi(z))  \|
			\leq \alpha(x, y, z)
		\end{multline}
		for all $x y, y \in X$. Then there exists a unique quadratic mapping $h: X \rightarrow Y$ such that:
		\begin{equation}\label{Eq-13}
			\|\phi(x)-h(x)\| \leq \varphi(x, x, 0) \quad, x \in X .
		\end{equation}
	\end{theorem}
	
	\begin{proof}
		Letting $x=y, z=0$ in \eqref{Eq-12}, we get:
		$$
		\|\phi(2 x)-4 \phi(x)\| \leq \alpha(x, x, 0)
		$$
		Hence 
		\begin{equation}\label{Eq-14-}
			\left\|\frac{1}{4} \phi(2 x)-\phi(x)\right\| \leq \frac{1}{4^\beta} \alpha(x, x, 0)
		\end{equation}
		Then by simple induction, we have
		\begin{equation}\label{Eq-15-}
			\left\|\frac{1}{4^k} \phi\left(2^k x\right)-\phi(x)\right\| \leq \sum_{j=1}^k \frac{1}{4^{\beta j}} \alpha\left(2^{j-1} x, 2^{j-1} x, 0\right)
		\end{equation}
		for all $x \in X$ and all positive integer $k$.\\
		For $k=1$, we obtain \eqref{Eq-14-}. 
		Suppose that \eqref{Eq-15-} holds for $k \in\mathbb{N}$. 
		We have.
		\begin{align*}
			\left\|\frac{1}{4^{k+1}} \phi\left(2^{k+1} x\right)-\phi(x)\right\| & =\frac{1}{4^{\beta}}\left\| \frac{\phi\left(2^k 2 x\right)}{4^k}-\phi(2 x) +  \phi(2 x)-4 \phi(x) \right\|\\
			& \leq \frac{1}{4^\beta} \sum_{j=1}^k \frac{1}{4^{\beta j}} \alpha\left(2^j x, 2^j x, 0\right)+\frac{1}{4^\beta} \alpha(x, x, 0) \\
			& =\sum_{j=1}^{k+1} \frac{1}{4^{\beta  j}} \alpha\left(2^{j-1} x, 2^{j-1} x, 0\right)
		\end{align*}
		Hence \eqref{Eq-15-} holds for every $k \in \mathbb{N}$.
		
		Let $n$ and $l$ be positive integers with $m>l$, by \eqref{Eq-15-},
		we obtain:
		\begin{align}
			\left\|  \frac{\phi\left(2^n x\right)}{4^n}  -\frac{\phi\left(2^l x\right)}{4^l}\right\|  = & \left\|  \frac{1}{4^l}\left(\frac{\phi\left(2^{n-l} \cdot 2^l x\right)}{4^{n-l}}-\phi\left(2^l x\right) \right)\right\|  \nonumber\\
			& \leq \frac{1}{4^{\beta l}} \sum_{j=1}^{n-l} \frac{1}{4^{\beta j}} \alpha\left(2^{j+l-1} x, 2^{j+l-1} x, 0\right) \nonumber\\
			& =\sum_{m=l+1}^n \frac{1}{4^{\beta m}} \alpha\left(2^{m-1} x, 2^{m-1} x, 0\right) \label{Eq-16}
		\end{align}
	\end{proof}
	It follows  from \eqref{Eq-16} and \eqref{Eq-4.1} that $\left\{\dfrac{\phi(2^n x)}{4^n}\right\}$ is a Cauchy sequence in $Y$ wich is Banach space, then there exists a mapping $h: X \rightarrow Y$ such that: 
	$$h(x)=\lim _{n \rightarrow \infty} \frac{\phi(2^n x)}{4^n} ;\qquad x \in X.$$
	Now, we take $l=0$ 
	and tender $n$ to $\infty$, we obtain \eqref{Eq-13}. 
	
	In order to verify that $h$ is quadratique, we write
	\begin{multline*}
		 \|h(x+y-z)+h(x+z-y)+h(y+z-x)-h(x-y)-h(x-z)-h(z-y)-h(x)-h(y)-h(z)\| \\
		 \leq \left\|h(x+y-z)-\frac{\phi\left(2^j(x+y- z)\right)}{4^j}\right\|+\left\|h(x+z-y)-\frac{\phi\left(2^j(x+z-y)\right)}{4^j}\right\|\\
		 +\left\|h(y+z-x)-\frac{\phi\left(2^j(y+z-x )\right)}{4^j}\right\| \\
		 +\left\|h(x-y)-\frac{\phi\left(2^j(x-y)\right)}{4^j}\right\|+\left\|h(x-z)-\frac{\phi\left(2^j(x-z)\right)}{4^j}\right\|+\left\|h(z-y)-\frac{\phi\left(2^j(z-y)\right)}{4^j}\right\| \\
		 +\left\|h(x)-\frac{\phi\left(2^j x\right)}{4^j}\right\|+\left\|h(y)-\frac{\phi(2^j y)}{4^j}\right\|+\left\|h(z)-\frac{\phi\left(2^j z\right)}{4 y^j}\right\| \\
		 +\| \frac{\phi\left(2^{j}(x+y-z)\right)}{4^j}+\frac{\phi\left(2^j(x+z-y)\right)}{4^j}+\frac{\phi\left(2^j(y+z-x)\right)}{4^y}-\frac{\phi\left(2^j(x-y)\right)}{4^j}-\frac{\phi\left(2^j(x-z)\right)}{4^j}\\
		 -\frac{\phi\left(2^j(z-y)\right)}{4^j}-\frac{\phi\left(2^{j} x\right)}{4^j}-\frac{\phi\left(2^j y\right)}{4^j}-\frac{\phi\left(2^j z\right)}{4 ^j} \| 
		 \rightarrow 0 \text { as } j \rightarrow \infty.
	\end{multline*}
	Then we get:
	$$
	h(x+y-z)+h(x+z-y)+h(y+z-x)-h(x-y)-h(x-z)-h(z-y)-h(x)-h(y)-h(z)=0
	$$
	for all $x, y, z \in X$. 
	By Lemma \ref{Lemma2.1}, we conclude that $h$ is quadratic.
	
	Finally, letting $Q$ another mapping sartisfying \eqref{Eq-13}. 
	We have
	\begin{align*}
		\|h(x)-Q(x)\| &\leq  \left\|\frac{h\left(2^k x\right)-\phi\left(2^k x\right)}{4^k}\right\|+\left\|\frac{Q\left(2^k x\right)-\phi\left(2^k x\right)}{4^k}\right\| \\
		&\leq  \frac{2}{4^{k\beta}} \varphi\left(2^k x, 2^k x, 0\right)\\
		& =2 \sum_{j=1}^{\infty} \frac{1}{4^{\beta(j+k)}} \alpha\left(2^{k+j-1}, 2^{k+j-1} x, 0\right) \\
		& =2 \sum_{l=k+1}^{\infty} \frac{1}{4^{\beta l}} \alpha\left(2^{l-1} x, 2^{l-1} x, 0\right) 
		 \rightarrow 0 \text { as } k \rightarrow \infty.
	\end{align*}
	Then we have $h=Q$.
	Now, we obtain a classical result of Ulam stability of \eqref{Eq-1}, by putting $\alpha=\varepsilon>0$.
	
	\begin{corollary}
		Let $X$ be a vector space, $Y$ be a $\beta$-homogeneous complex Banach space with $0<\beta \leq 1$. Let $\phi: X \rightarrow Y$ be a mapping that satisfies  $\phi(0)=0$ and
		$$\|\phi(x+y-z)+\phi(x+z-y)+\phi(y+z-x)-\phi(x-y)-\phi(x-z)-\phi(z-y)-\phi(x)-\phi(y)-\phi(z)\| \leq \varepsilon$$ 
		for all $x, y, z \in X$. 
		Then there exists a unique quadratic mapping $h: X \rightarrow Y$ such that
		$$
		\|\phi(x)-h(x)\| \leq \frac{\varepsilon}{4^\beta-1} \quad ;\qquad  x \in X .
		$$	
	\end{corollary}
	
\end{document}